\documentclass[12pt,oneside]{amsart}
\usepackage{amssymb, amsmath, amsthm}

\usepackage{epsfig}
\usepackage{epstopdf}
\usepackage{graphicx}
\usepackage{color}
\usepackage{enumerate}
\usepackage{amsrefs}
\usepackage{array}

\usepackage[normalem]{ulem}

\usepackage[normalem]{ulem}

\usepackage{pinlabel}

\theoremstyle{plain}
\newtheorem{theorem}{Theorem}[section]

\newtheorem*{theorem*}{Theorem}
\newtheorem*{theorem-DisjtSS}{Theorem \ref{Thm: Disjt SS}}
\newtheorem*{theorem-EssentialTorus}{Theorem \ref{Thm: Essential Torus}}
\newtheorem*{cor-ScharlWu}{Corollary \ref{Cor-ScharlWu}}
\newtheorem*{corollary-MSC 2}{Corollary \ref{Cor: MSC 2}}
\newtheorem*{theorem-Main A}{Theorem \ref{Thm: Main A}}
\newtheorem*{theorem-Main B}{Theorem \ref{Thm: Main Thm B}}
\newtheorem*{Cor-Unknotting}{Theorem \ref{Cor: Prime Unknotting 1}}
\newtheorem*{Cor-genusbandsum}{Theorem \ref{Thm: Genus superadd}}
\newtheorem*{Cor-bandsumscc}{Corollary \ref{Cor: Band Sums CC}}

\newtheorem{definition}[theorem]{Definition}

\theoremstyle{definition}

      \makeatletter
      \def\@setcopyright{}
      \def\serieslogo@{}
      \makeatother

\begin{document}

   \title[Three-dimensional rep-tiles]{Three-dimensional rep-tiles}
   \author{Ryan Blair, Zoe Marley and Ilianna Richards}

\begin{abstract}
A 3D rep-tile is a compact 3-manifold $X$ in $\mathbb{R}^3$ that can be decomposed into finitely many pieces, each of which are similar to $X$, and all of which are congruent to each other. In this paper we classify all 3D rep-tiles up to homeomorphism. In particular, we show that a 3-manifold is homeomorphic to a 3D rep-tile if and only if it is the exterior of a connected graph in $S^3$. 
\end{abstract}

\maketitle
\date{\today}

\section{Introduction}

A set $X \subset \mathbb{R}^n$ with non-empty interior is an $n$-dimensional $k$-index rep-tile if there are sets $X_1,X_2,..., X_k$ with disjoint interiors and with $X= \cup_{i=1}^{k} X_i$ that are mutually congruent and similar to $X$. Rep-tiles have been well-studied since 1963, when they were introduced by Gardner and Golomb \cite{Gardner63}, \cite{Golomb64}. Much of the study of rep-tiles has focused on 2-dimensional rep-tiles from the perspective of fractal geometry. There are robust methods for constructing rep-tiles \cite{B91}. However, these methods usually result in rep-tiles that are not piecewise linear manifolds. Various authors have explored the basic topological properties of rep-tiles. For example, a rep-tile $X$ has a \emph{hole} if the complement of the closure of some component of the interior of $X$ has a bounded component. John Conway asked if there exists a  $2$-dimensional rep-tile with a hole, and Gr{\"u}mbaum answered this question in the affirmative by providing a $36$-index rep-tile example \cite{Croft91} (also see \cite{JN05}). The topology of $2$-dimensional rep-tiles is limited. In particular, Bandt and Wang \cite{BW01} and Luo et al. \cite{LRT} proved that if the interior of a $2$-dimensional rep-tile is connected, then the rep-tile is a topological disk. 

In this paper we apply $3$-manifold techniques to study the topology of 3-dimensional rep-tiles. To this end, we restrict to \emph{3D rep-tiles},  $3$-dimensional $k$-index rep-tiles that are piecewise linear embeddings of compact connected 3-manifolds in $\mathbb{R}^3$ with $k>1$. Examples of 3D rep-tiles that are homeomorphic to the 3-ball are prevalent. For example, a cube is an 8-index 3D rep-tile homeomorphic to a 3-ball, see Figure \ref{IdempVSRep}. Additionally, various authors have investigated which tetrahedra are 3D rep-tiles  \cite{LJ94}, \cite{MS11}, \cite{KP17}, \cite{Haverkort2018}. However, examples homeomorphic to other 3-manifolds have been more challenging to generate. In 1998, Goodman-Strauss asked if there exists a 3D rep-tile with a hole. This was answered by Van Ophysen, who generated an example homeomorphic to a solid torus \cite{Prob19}. See Figure \ref{G1Examp} for an example of a 3D rep-tile homeomorphic to a solid torus. Furthermore, a description of an infinite family of 3D rep-tiles homeomorphic to any genus $n$ handlebody is given in \cite{Prob19}. In \cite{BM20}, the authors use a computer search to generate an 8-index 3D rep-tile  homeomorphic to a solid torus. This was the unique 8-index 3D rep-tile not homeomorphic to the 3-ball among the one million examples they generated.  

In this paper we classify all 3D rep-tiles up to homeomorphism. In particular, we give a method of constructing a 3D rep-tile of any suitable homeomorphism type.

\begin{theorem}\label{main}
A set $X$ in $\mathbb{R}^3$ is a $3D$ rep-tile if and only if $X$ is homeomorphic to the exterior of a finite connected graph in $S^3$.
\end{theorem}

The concept of self-affine tile is related to that of rep-tile. A self-affine tile is defined by a finite collection of contractions (not necessarily similarities) that are affine translates of a single linear contraction. In \cite{CT16}, the authors generate the first examples of 3-dimensional self-affine tiles which are topological 3-manifolds with boundary. In particular, they show that every handlebody is homeomorphic to a 3-dimensional self-affine tile. However, the embeddings they generate are not piece-wise linear and not 3D rep-tiles. They conjecture that all 3-dimensional self-affine tiles which are topological 3-manifolds are homeomorphic to handlebodies. We wonder if the examples of rep-tiles which are not homeomorphic to handlebodies generated in the current paper can be modified to provide counterexamples to the conjecture posed in \cite{CT16}.

Rep-tiles are a model for topological self-replication in that a rep-tile is an object that can be decomposed into finitely many copies each of which is homeomorphic to the original. Previous models of self-replication in low-dimensional topology have focused on idempotents (i.e. morphisms with the property that $f\circ f=f$) in the appropriate topological category. Such idempotents are manifolds that can be decomposed along embedded surfaces into two copies of themselves. In particular, idempotents have been classified in the Temperley-Lieb category \cite{Abramsky07}, the tangle category \cite{BS19}, and the (2+1)-cobordism category \cite{BL21}. Models based on idempotents inherently model 1-to-2 self-replication. In contrast, $3D$ rep-tiles model 1-to-many self-replication. See Figure \ref{IdempVSRep}.

\begin{figure}[h!]
\begin{picture}(264,267)
\put(1,1){\includegraphics[scale=.5]{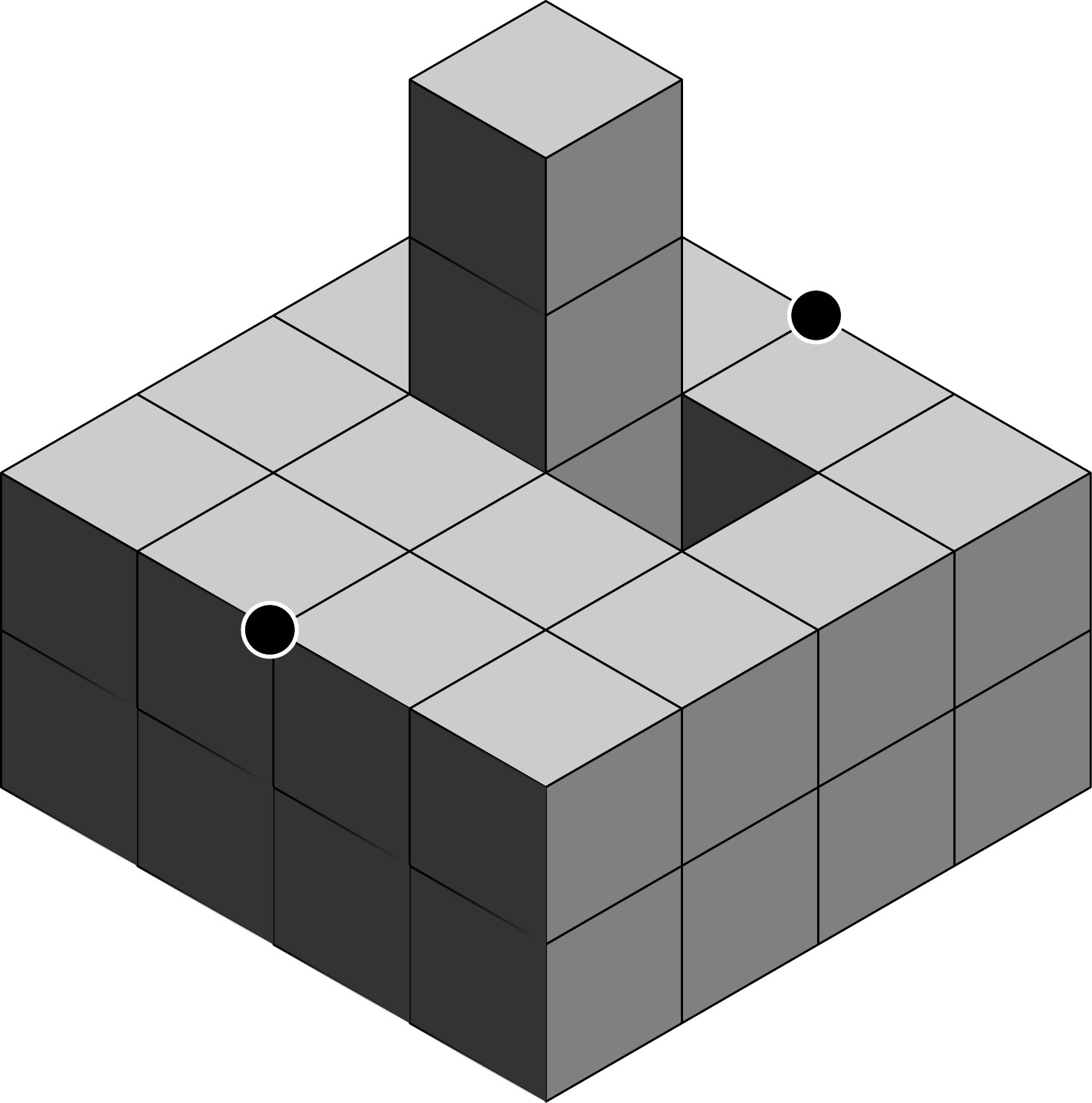}}
\put(49,100){$A$}
\put(157,162){$B$}
\end{picture}
\caption{An example of a polycube and a 3D rep-tile homeomorphic to $D^2\times S^1$. It is constructed by removing the two cubes corresponding to $[-1,0]\times[0,1]\times[0,2]$ from $[-2,2]\times[-2,2]\times[0,2]$ and then adding two cubes corresponding to $[-1,0]\times[-1,0]\times[2,4]$.}
\label{G1Examp}
\end{figure}

\begin{figure}[h!]
\begin{picture}(277,198)
\put(1,1){\includegraphics[scale=.35]{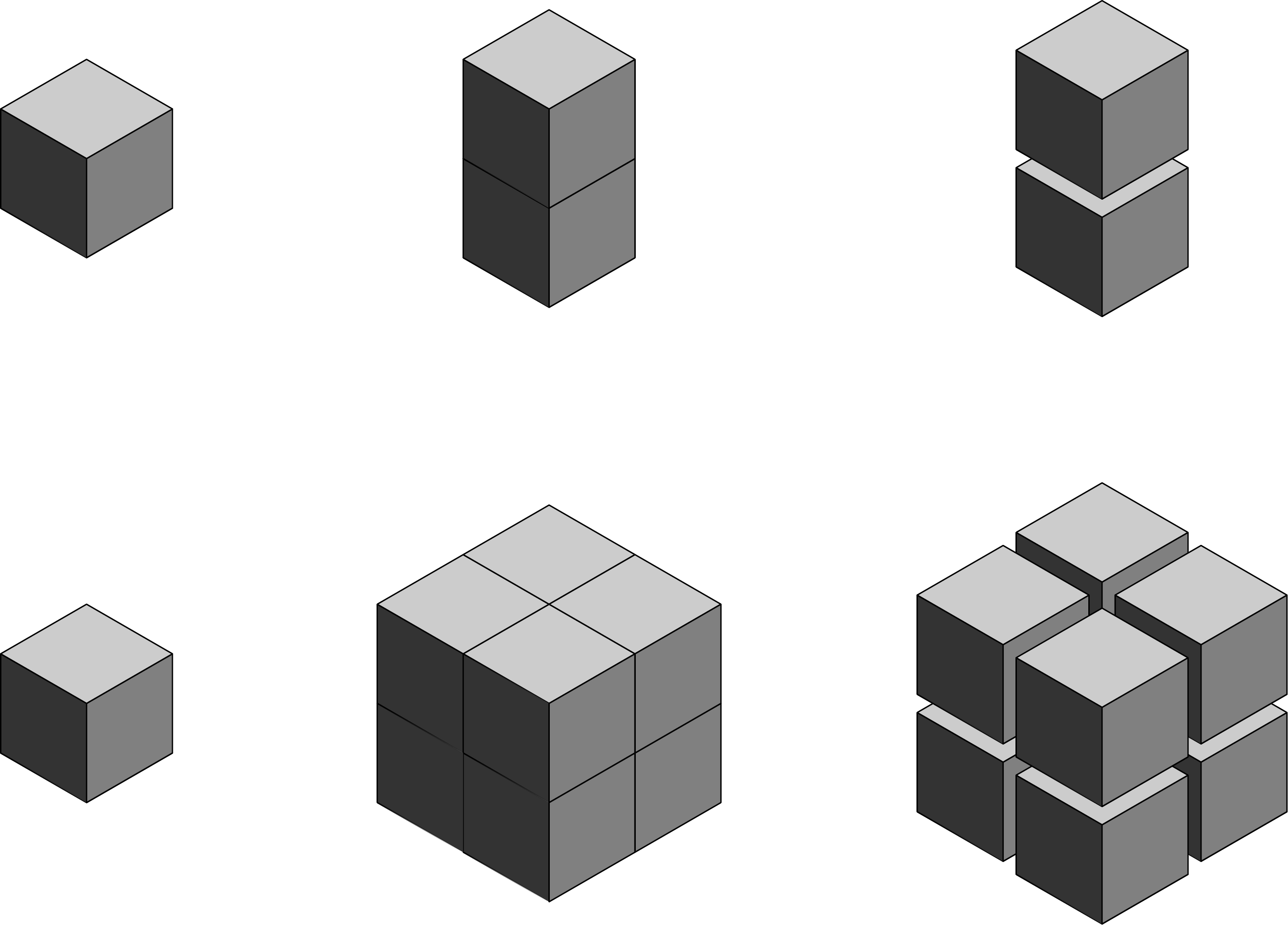}}
\put(55,45){$\cong$}
\put(65,163){$\cong$}
\put(169,45){$\rightarrow$}
\put(175,163){$\rightarrow$}
\end{picture}
\caption{Above: A cube as an idempotent in the (2+1)-cobordism category modeling 1-to-2 self-replication. Below: A cube as a  $3$-dimensional $8$-index rep-tile modeling 1-to-8 self-replication.}
\label{IdempVSRep}
\end{figure}
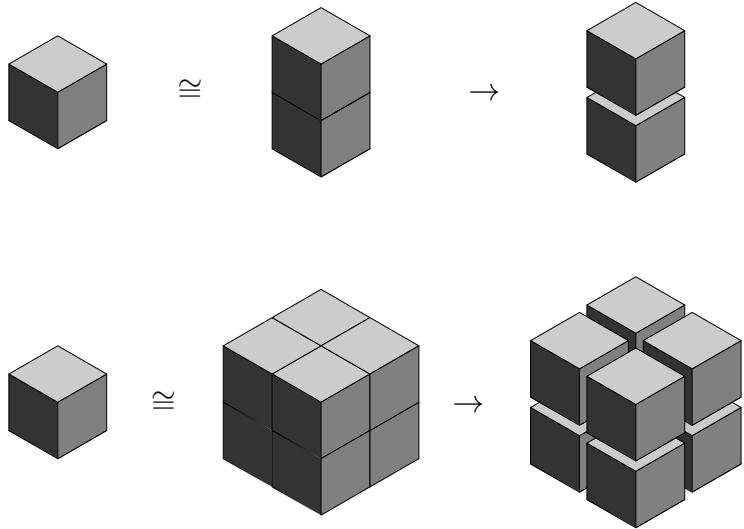

In Section \ref{defs} we introduce useful 3-manifold and lattice terminology. In Section \ref{Sec:backward} we prove the backward dirrection of Theorem \ref{main}. In Section \ref{Sec:forward} we prove the forward dirrection of Theorem \ref{main}.

\section{Preliminaries}\label{defs}



We begin with a discussion of useful classes of manifolds and sub-manifolds. A \emph{handlebody} is a compact orientable $3$-manifold with boundary homeomorphic to the closed regular neighborhood of a connected finite graph embedded in a 3-manifold. The genus of the boundary of a handlebody uniquely determines the homeomorphism class of the handlebody. We say a handlebody is genus $g$ if its boundary is homeomorphic to a genus $g$ surface.

\begin{definition}
Given manifold $M$ with submanifold $N$, the \emph{exterior} of $N$ in $M$ is $E_M (N):=M \setminus \eta(N)$ where $\eta(N)$ is an open tubular neighborhood of $N$ in $M$. When $M$ is understood to be $S^3$, we will write the exterior as $E(N)$. 
\end{definition}

Here we introduce terminology related to cubic lattices in $\mathbb{R}^3$ that will be relevant to our construction. We will denote the standard embedding of the unit cubic lattice in $\mathbb{R}^3$ as $$\mathcal{Z}^3= \bigcup_{(i,j)\in \mathbb{Z}^2} (\{i\}\times \{j\} \times \mathbb{R}) \cup (\{i\}\times \mathbb{R} \times \{j\}) \cup (\mathbb{R} \times  \{i\} \times \{j\}).$$ Suppose $\lambda>0$ and let $f_{\lambda}:\mathbb{R}^3\rightarrow \mathbb{R}^3$ denote the scaling function given by $f(x)=\lambda x$. Let $\mathcal{Z}_{\lambda}^3=f(\mathcal{Z}^3)$. 


Sometimes it will be useful to refer to the 3-complex structure induced by each of the cubic lattices mentioned above on $\mathbb{R}^3$ (i.e. the decomposition $\mathbb{R}^3$ into the union of vertices, edges, squares and cubes). Given the lattice $\mathcal{Z}_{\lambda}^3$, the corresponding 3-complex structure on $\mathbb{R}^3$ will be denoted $\mathcal{C}(\mathcal{Z}_{\lambda}^3)$. Ultimately, our construction of 3D rep-tiles will be based on polycubes. A \emph{polycube} is a subset of $\mathbb{R}^3$ that is similar to a finite union of $3$-cells in $\mathcal{C}(\mathcal{Z}^3)$. See Figure \ref{G1Examp} for an example of a polycube consisting of 32 cubes and homeomorphic to $D^2\times S^1$. 

\section{Graph exteriors as 3D rep-tiles}\label{Sec:backward}

In \cite{Adams97} Adams shows that for any 3-dimensional compact submanifold $M$ of $\mathbb{R}^3$ with a single boundary component, a $3$-ball can be decomposed into four congruent tiles, each homeomorphic to $M$. The following theorem generalizes Adams’ results and then applies the generalization to 3D rep-tiles. In particular, the following proof implies that for any 3-dimensional compact submanifold $M$ of $\mathbb{R}^3$ with a single boundary component, a cube can be decomposed into two congruent tiles such that each is a polycube and each is homeomorphic to $M$. For example, two copies of the polycube in Figure \ref{G1Examp} tile the $4\times 4\times 4$ cube. We can see this by rotating the polycube in Figure \ref{G1Examp} by $\pi$ about the line that passes through the points $A$ and $B$. This observation implies that this polycube is a rep-tile since each of its 32 cubes can be decomposed into a total of 64 polycubes which are all congruent to each other and similar to the original.

\begin{theorem}\label{forward}
If $M$ is the exterior of a connected graph in $S^3$, then $M$ is homeomorphic to a 3D rep-tile.
\end{theorem}

\begin{proof}
Let $M=E(\Gamma)$ for some connected graph $\Gamma$ embedded in $S^3$. We begin by showing that $E(\Gamma)$ is homeomorphic to a \emph{cube-with-holes} (i.e. the exterior of a collection of properly embedded arcs in a cube). Preform edge contractions on $\Gamma$ to produce an embedded graph $\Gamma_1$ with a single vertex and $n$ edges. Note that $E(\Gamma)$ is homeomorphic to $E(\Gamma_1)$. See Figure 3 for an example of an embedded graph in $S^3$ with a single vertex. Let $V$ be a $3$-ball corresponding to a small closed regular neighborhood of the vertex of $\Gamma_1$. Let $B$ be the 3-ball $S^3\setminus int(V)$. The 3-ball $B$ meets in $\Gamma_1$ in a collection of $n$ properly embedded arcs $\alpha_1, \alpha_2, ..., \alpha_n$. Note that $E_B(\cup_{i=1}^n \alpha_i)$ is homeomorphic to $E(\Gamma_1)$. Label the boundary of each arc by $\partial \alpha_i = \{a_i, b_i\}$. Identify the three ball $B$ with the unit cube $C= [0,1]\times [0,1] \times [0,1]$ in $\mathbb{R}^3$ via a homeomorphism that takes the set $\{a_1, a_2,...,a_n\}$ to the interior of the square $[0,1]\times [0,1] \times \{0\}$ and takes the set $\{b_1, b_2,...,b_n\}$ to the interior of the square $[0,1]\times [0,1] \times \{1\}$. See Figure 4. Through an abuse of notation, we continue to refer to the image of the arcs properly embedded in $C$ as $\alpha_1, \alpha_2, ..., \alpha_n$. Thus, $E_C(\cup_{i=1}^n \alpha_i) \cong E_B(\cup_{i=1}^n \alpha_i)$. 

\begin{figure}[h!]\label{Fig:Graph}
\includegraphics[scale=1]{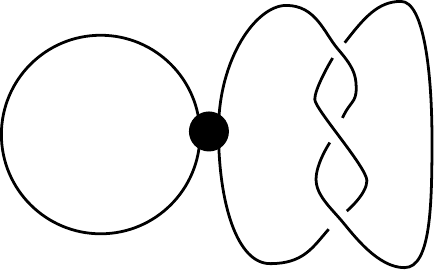}
\caption{A graph with a single vertex embedded in $S^3$.}
\end{figure}

\begin{figure}[h!]\label{CubeWithArcs}
\includegraphics[scale=.15]{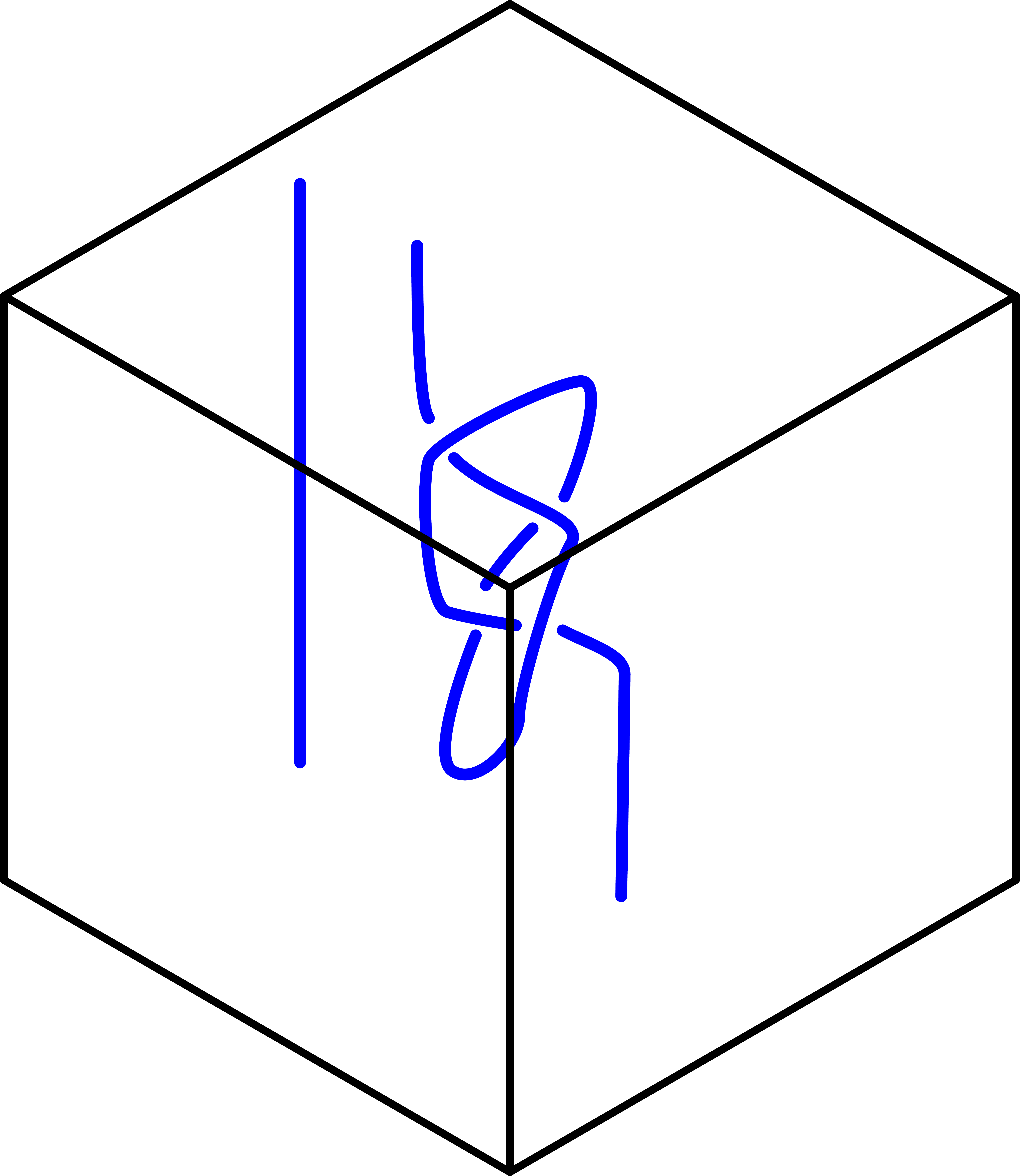}
\caption{Arcs $\alpha_1$ and $\alpha_2$ properly embedded in the cube $C$ such that $E_{C}(\alpha_1\cup \alpha_2)$ is homeomorphic to the exterior of the graph in Figure 3.}
\end{figure}

\begin{figure}[h!]\label{CubeWithHoles}
\includegraphics[scale=.15]{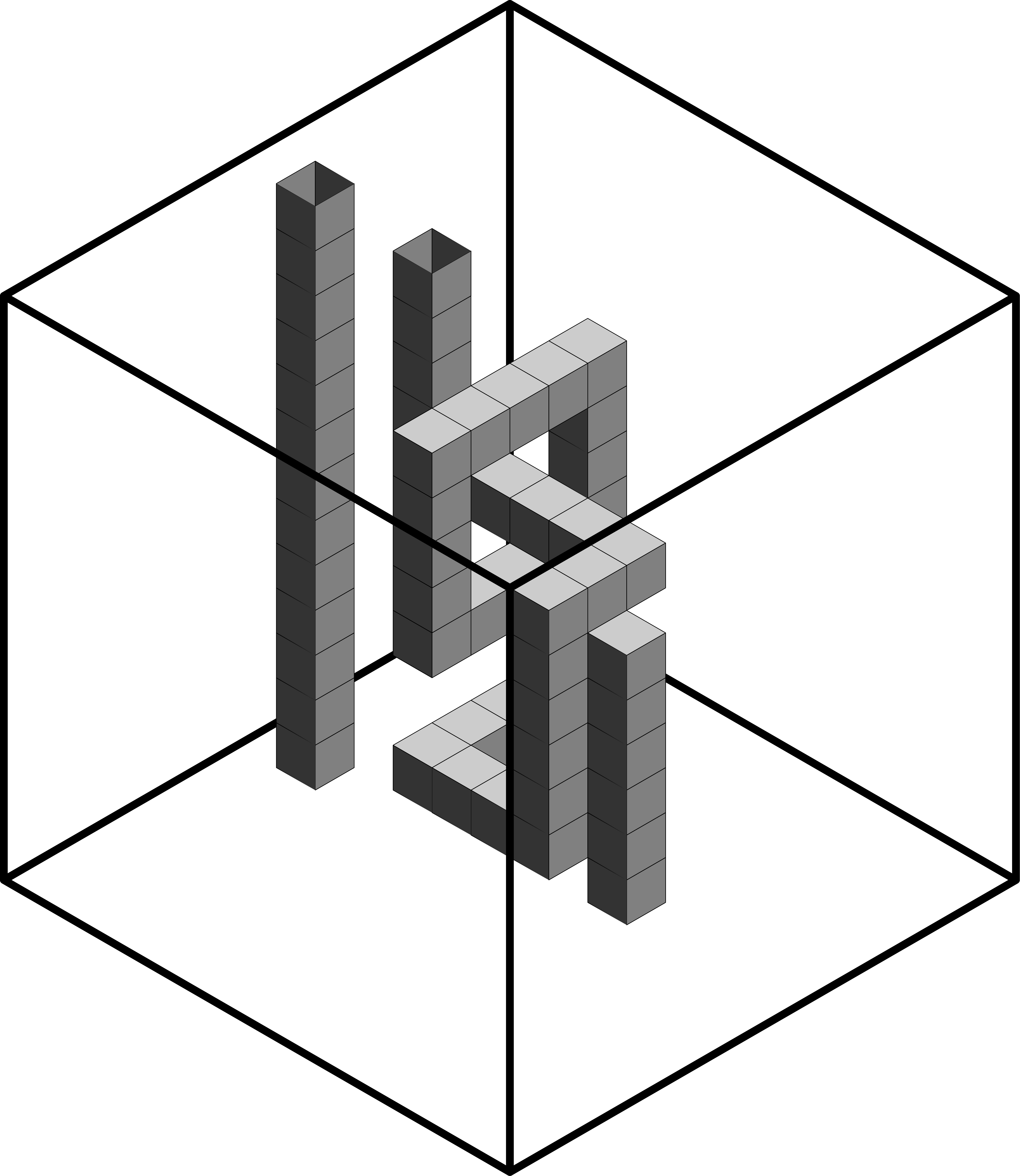}
\caption{The exterior of the graph in Figure 3 represented as a polycube.}
\end{figure}

Next, we show the cube-with-holes we just generated is homeomorphic to a polycube. Note that for every positive integer $m$, $\mathcal{C}(\mathcal{Z}_{\frac{1}{m}}^3)$ induces a cell structure on $C$. Moreover, by choosing $m$ sufficiently large and possibly preforming an small proper isotopy of  $\cup_{i=1}^n \alpha_i$, we can assume $\cup_{i=1}^n \alpha_i$ is disjoint from the $\mathcal{Z}_{\frac{1}{m}}^3$ and, if $A$ is the union of all 3-cells in $C$ with non-trivial intersection with $\cup_{i=1}^n \alpha_i$, then $A$ is isotopic to a closed tubular neighborhood of $\cup_{i=1}^n \alpha_i$ in $C$. Let $i(A)$ denote the interior of all cells in $C$ with non trivial intersection with $\cup_{i=1}^n \alpha_i$. Then $i(A)$ is isotopic to an open tubular neighborhood of $\cup_{i=1}^n \alpha_i$ in $C$ and $C\setminus i(A)$ is a polycube homeomorphic to $E_C(\cup_{i=1}^n \alpha_i)\cong E(\Gamma)$. See Figure 5. 

We now modify $C\setminus i(A)$ to generate a polycube homeomorphic to $C\setminus i(A)$ which tiles the cube using two congruent copies. Let $C_1=[0,1]\times [-1,0] \times [0,1]$, $C_2=[-1,0]\times [-1,0] \times [0,1]$ and $C_3=[-1,0]\times [0,1] \times [0,1]$. Note that $(C\setminus i(A) )\cup C_1 \cup C_2 \cup C_3$ is again a polycube homeomorphic to $E(\Gamma)$. Let $r:\mathbb{R}^3\rightarrow \mathbb{R}^3$ be a rotation of $\pi$ about the line parameterizd by $<t,0,1>$. Let $X= (C\setminus i(A) )\cup (C_1\cup C_2 \cup C_3)\cup r(A)$. The polycube $X$ is homeomorphic to $(C\setminus i(A) )\cup C_1 \cup C_2 \cup C_3$ since it is the boundary connected sum of $(C\setminus i(A) )\cup C_1 \cup C_2 \cup C_3$ with a collection of $n$ 3-balls. See Figure 6. Hence, $X\cong E(\Gamma)$. By construction, $X\cup r(X)= [-1,1]\times [-1,1] \times [0,2]$ and $int(r(X))\cap int(X)=\emptyset$. Hence, two congruent copies of $X$ with disjoint interiors tile the cube $[-1,1]\times [-1,1] \times [0,2]$. Since $X$ is a polycube consisting of $4m^3$ cubes of side length $\frac{1}{m}$, then $X$ is an $3$-dimensional $8m^3$-index  rep-tile.

\end{proof}

\begin{figure}[h!]\label{ExRep1}
\includegraphics[scale=.08]{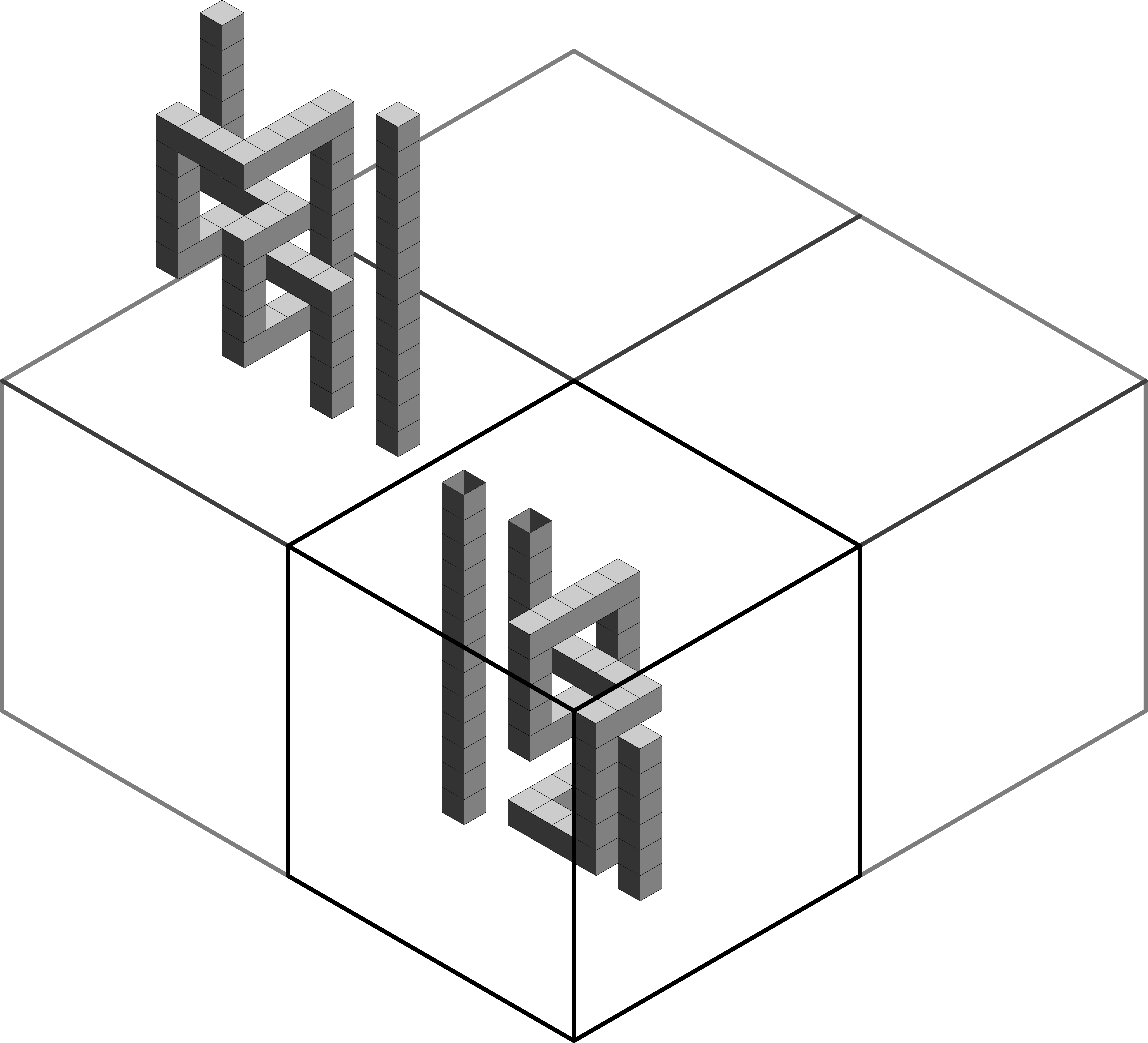}
\caption{A 3D rep-tile homeomorphic to the exterior of the graph in Figure 3.}
\end{figure}

\section{Classification}\label{Sec:forward}

In order to show that every 3D rep-tile is homeomorphic to the exterior of a connected graph, we will require the following reimbedding theorem due to Fox.

\begin{theorem}\label{Fox}\cite{Fox48}
Every compact connected 3-dimensional sub-manifold $X$ of $S^3$ can be reimbedded in $S^3$ so that the exterior of the image of $X$ is a union of handlebodies, i.e. regular neighborhoods of embedded graphs.
\end{theorem}

The idea behind the following proof is to show that if a 3D rep-tile $X= \bigcup_{i=1}^{k} X_i$ has two boundary components, then a ``minimal'' boundary component over all $X_i$ must also be a boundary component for $X$. This leads to a contradiction.

\begin{theorem}\label{boundaryconnected}
If $X$ is a 3D rep-tile, then $\partial X$ is connected.
\end{theorem}

\begin{proof}
Suppose $X$ is a 3D rep-tile such that $\partial X$ is not connected. Then $X= \bigcup_{i=1}^{k} X_i$ such that, $X_i$ is congruent to $X_j$ for any $i, j \in \{1,..., k\}$, $int(X_i) \cap int(X_j) = \emptyset$ if $i \neq j$, and $X_1$ is similar to $X$. Moreover, each $X_i$ has at least two connected boundary components.

Note that every closed connected surface in $\mathbb{R}^3$ separates. Hence, any such surface $G$ is the boundary of some finite volume region $R_G\subset \mathbb{R}^3$. Let $F$ be a connected boundary component of some $X_m$ such that $R_F$ has the least volume over all finite volume regions bounded by any connected boundary component of any $X_i$. 



Suppose the interior of $R_F$ has nontrivial intersection with $X$. Then there exists some $X_n$ embedded in $R_F$. Since $X_n$ has at least two boundary components and $R_F$ has one boundary component, then $X_n\neq R_F$. So, some boundary component of $X_n$ bounds a finite volume region which is a proper subset of $R_F$ and has volume strictly less than $R_F$. This is a contradiction to the minimality of the volume of $R_F$.  Thus, the interior of $R_F$ is disjoint from $X$. This implies that $F$ is a boundary component of $X$. 

However, if $X= \bigcup_{i=1}^{k} X_i$ is a rep-tile and $v$ is the volume of $R_F$, then the least volume of any finite volume region bounded by a connected boundary component of $X$ is $kv$. Since $F$ is a boundary component of $X$, then $kv=v$. Since $v>0$, then $k=1$, a contradiction. Hence, $\partial X$ must be connected.
\end{proof}

Theorems \ref{Fox} and \ref{boundaryconnected} immediately imply the following result.

\begin{theorem}\label{backward}
Every 3D rep-tile is homeomorphic to the exterior of a connected graph.
\end{theorem}

Theorem \ref{main} follows from Theorems \ref{forward} and \ref{backward}.

\vspace{.5cm}

\noindent \textbf{Acknowledgements.} All three authors were partially supported by NSF grant DMS-1916494. The first author was partially supported by NSF grant DMS-1821254. We would like to thank Florence Newberger for helpful conversations and Chaim Goodman-Strauss for helpful comments on an early draft.

\bibliographystyle{plain}
\bibliography{bib}

\begin{bibdiv}
\begin{biblist}

\bib{Abramsky07}{incollection}{
      author={Abramsky, Samson},
       title={Temperley-lieb algebra: from knot theory to logic and computation
  via quantum mechanics},
        date={2007},
   booktitle={Mathematics of quantum computation and quantum technology},
   publisher={Taylor and Francis},
       pages={515\ndash 558},
}

\bib{Adams97}{incollection}{
      author={Adams, Colin~C.},
       title={Knotted tilings},
        date={1997},
   booktitle={The mathematics of long-range aperiodic order ({W}aterloo, {ON},
  1995)},
      series={NATO Adv. Sci. Inst. Ser. C Math. Phys. Sci.},
      volume={489},
   publisher={Kluwer Acad. Publ., Dordrecht},
       pages={1\ndash 8},
      review={\MR{1460017}},
}

\bib{BW01}{article}{
      author={Bandt, C.},
      author={Wang, Y.},
       title={Disk-like self-affine tiles in {$\Bbb R^2$}},
        date={2001},
        ISSN={0179-5376},
     journal={Discrete Comput. Geom.},
      volume={26},
      number={4},
       pages={591\ndash 601},
         url={https://doi.org/10.1007/s00454-001-0034-y},
      review={\MR{1863811}},
}

\bib{B91}{article}{
      author={Bandt, Christoph},
       title={Self-similar sets. {V}. {I}nteger matrices and fractal tilings of
  {${\bf R}^n$}},
        date={1991},
        ISSN={0002-9939},
     journal={Proc. Amer. Math. Soc.},
      volume={112},
      number={2},
       pages={549\ndash 562},
         url={https://doi.org/10.2307/2048752},
      review={\MR{1036982}},
}

\bib{BM20}{article}{
      author={Bandt, Christoph},
      author={Mekhontsev, Dmitry},
       title={Computer geometry: rep-tiles with a hole},
        date={2020},
        ISSN={0343-6993},
     journal={Math. Intelligencer},
      volume={42},
      number={1},
       pages={1\ndash 5},
         url={https://doi.org/10.1007/s00283-019-09923-6},
      review={\MR{4068835}},
}

\bib{BL21}{article}{
      author={Blair, Ryan},
      author={Lee, Ricky},
       title={Self-replicating 3-manifolds},
     journal={arXiv:2107.04528},
}

\bib{BS19}{article}{
      author={Blair, Ryan},
      author={Sack, Joshua},
       title={Idempotents in tangle categories split},
        date={2019},
        ISSN={0218-2165},
     journal={J. Knot Theory Ramifications},
      volume={28},
      number={5},
       pages={1950025, 9},
         url={https://doi.org/10.1142/S0218216519500251},
      review={\MR{3943699}},
}

\bib{CT16}{article}{
      author={Conner, Gregory~R.},
      author={Thuswaldner, J\"{o}rg~M.},
       title={Self-affine manifolds},
        date={2016},
        ISSN={0001-8708},
     journal={Adv. Math.},
      volume={289},
       pages={725\ndash 783},
         url={https://doi.org/10.1016/j.aim.2015.11.022},
      review={\MR{3439698}},
}

\bib{Croft91}{book}{
      author={Croft, HT},
      author={Falconer, KJ},
      author={Guy, RK},
       title={Unsolved problems in geometry},
    language={English},
   publisher={Springer},
        date={1991},
        ISBN={038795063},
}

\bib{Fox48}{article}{
      author={Fox, Ralph~H.},
       title={On the imbedding of polyhedra in {$3$}-space},
        date={1948},
        ISSN={0003-486X},
     journal={Ann. of Math. (2)},
      volume={49},
       pages={462\ndash 470},
         url={https://doi.org/10.2307/1969291},
      review={\MR{26326}},
}

\bib{Gardner63}{article}{
      author={Gardner, Martin},
       title={On rep-tiles, polygons that can make larger and smaller copies of
  themselves},
        date={1963},
     journal={Scientific Amer.},
      number={208},
       pages={154\ndash 164},
}

\bib{Golomb64}{article}{
      author={Golomb, S.~W.},
       title={Replicating figures in the plane},
        date={1964},
     journal={Math.Gaz.},
      number={48},
       pages={403\ndash 412},
}

\bib{Haverkort2018}{article}{
      author={Haverkort, Herman},
       title={No acute tetrahedron is an 8-reptile},
        date={2018},
        ISSN={0012-365X},
     journal={Discrete Math.},
      volume={341},
      number={4},
       pages={1131\ndash 1135},
         url={https://doi.org/10.1016/j.disc.2017.10.010},
      review={\MR{3764364}},
}

\bib{JN05}{article}{
      author={Jordan, Francis},
      author={Ngai, Sze-Man},
       title={Reptiles with holes},
        date={2005},
        ISSN={0013-0915},
     journal={Proc. Edinb. Math. Soc. (2)},
      volume={48},
      number={3},
       pages={651\ndash 671},
         url={https://doi.org/10.1017/S001309150400001X},
      review={\MR{2171191}},
}

\bib{KP17}{article}{
      author={Kyn\v{c}l, Jan},
      author={Pat\'{a}kov\'{a}, Zuzana},
       title={On the nonexistence of {$k$}-reptile simplices in {$\Bbb R^3$}
  and {$\Bbb R^4$}},
        date={2017},
     journal={Electron. J. Combin.},
      volume={24},
      number={3},
       pages={Paper No. 3.1, 44},
      review={\MR{3691518}},
}

\bib{LJ94}{article}{
      author={Liu, Anwei},
      author={Joe, Barry},
       title={On the shape of tetrahedra from bisection},
        date={1994},
        ISSN={0025-5718},
     journal={Math. Comp.},
      volume={63},
      number={207},
       pages={141\ndash 154},
         url={https://doi.org/10.2307/2153566},
      review={\MR{1240660}},
}

\bib{LRT}{article}{
      author={Luo, Jun},
      author={Rao, Hui},
      author={Tan, Bo},
       title={Topological structure of self-similar sets},
        date={2002},
        ISSN={0218-348X},
     journal={Fractals},
      volume={10},
      number={2},
       pages={223\ndash 227},
         url={https://doi.org/10.1142/S0218348X0200104X},
      review={\MR{1910665}},
}

\bib{MS11}{article}{
      author={Matou\v{s}ek, Ji\v{r}\'{\i}},
      author={Safernov\'{a}, Zuzana},
       title={On the nonexistence of {$k$}-reptile tetrahedra},
        date={2011},
        ISSN={0179-5376},
     journal={Discrete Comput. Geom.},
      volume={46},
      number={3},
       pages={599\ndash 609},
         url={https://doi.org/10.1007/s00454-011-9334-z},
      review={\MR{2826971}},
}

\bib{Prob19}{article}{
      author={van Ophuysen, G.},
       title={Problem 19},
        date={1997},
     journal={Tagungsbericht},
      number={20},
}

\end{biblist}
\end{bibdiv}

\end{document}